\documentclass[a4paper,11pt,reqno]{amsart}
\usepackage{amsmath,amssymb,amsfonts,amsthm}
\usepackage{latexsym,textcomp}
\usepackage{dsfont}
\usepackage[german, english]{babel}
\usepackage{hyperref}
\usepackage{enumerate}
\usepackage[utf8]{inputenc}

\newtheorem{lemma}{Lemma}[section]
\newtheorem{theorem}[lemma]{Theorem}

\newtheorem{proposition}[lemma]{Proposition}
\newtheorem{corollary}[lemma]{Corollary}

\theoremstyle{definition}

\newtheorem{example}[lemma]{Example}
\newtheorem*{remark}{Remark}

\numberwithin{equation}{section}
\newcommand{\comment}[1]{}

\newcommand{\cL}{{\mathcal L}}

\newcommand{\cQ}{{\mathcal Q}}
\newcommand{\cD}{{\mathcal D}}

\newcommand{\ds}{\delta_\sigma}

\newcommand{\R}{{\mathbb R}}

\newcommand{\N}{{\mathbb N}}

\newcommand{\supp}{{\mathrm {supp}\,}}

\newcommand{\av}[1]{\left\Vert #1\right\Vert}
\newcommand{\aV}[1]{\Vert #1 \Vert}

\newcommand{\Hm}[1]{\leavevmode{\marginpar{\tiny%
$\hbox to 0mm{\hspace*{-0.5mm}$\leftarrow$\hss}%
\vcenter{\vrule depth 0.1mm height 0.1mm width \the\marginparwidth}%
\hbox to 0mm{\hss$\rightarrow$\hspace*{-0.5mm}}$\\\relax\raggedright
#1}}}

\begin{document}

\title{Harris' criterion and Hardy inequalites on graphs}

\author{Simon Murmann}
\address{Institut für Mathematik, Friedrich-Schiller-Universität Jena, 07743 Jena, Germany.} \email{simon.murmann@uni-jena.de}

\author{Marcel Schmidt}
\address{Mathematisches Institut, Universität Leipzig, 04109 Leipzig, Germany.} \email{marcel.schmidt@math.uni-leipzig.de}

\maketitle

\begin{abstract}
In this paper we give a version of Harris' criterion for determining $H^{1,p}_0$ within $H^{1,p}$ on discrete spaces. Moreover, we provide a converse via Hardy inequalities involving distances to metric boundaries.   
\end{abstract}




\section{Introduction}

Sobolev spaces can be defined over various geometric spaces that admit a suitable notion of weak derivatives. Among them are Riemannian manifolds, discrete graphs, metric graphs and fractals - just to name a few. We are mostly interested in rough geometries, hence  this text focuses on first order Sobolev spaces.  Given a geometric space $X$ as discussed above the first order Sobolev spaces come in pairs: 
\begin{itemize}
 \item  $H^{1,p}(X)$ - the space of weakly differentiable $L^p$-functions with gradient in $L^p$.
 \item $H_0^{1,p}(X)$ - the closure of compactly supported $H^{1,p}(X)$-functions (or even smoother functions) in $H^{1,p}(X)$.
\end{itemize}
Typically $H_0^{1,p}(X)$ can be thought of being the space of those functions in $H^{1,p}(X)$ that vanish at 'the boundary' of $X$, in the sense that their trace is zero. While this statement can be made precise in great generality by choosing an appropriate notion of boundary and  trace functional, in general it is still hard to determine when a function from $H^{1,p}(X)$ actually belongs to $H^{1,p}_0(X)$. On domains in Euclidean space  there is a convenient criterion due to Harris: Let $\emptyset \neq \Omega \subseteq \R^n$ be a bounded domain and let
$$\delta \colon \Omega \to \R, \quad  \delta(x) = {\rm dist}(x,\partial \Omega)$$
denote the distance to the boundary. Then for $1 < p < \infty$  a function $f \in H^{1,p}(\Omega)$ with $f/\delta \in  L^p(\Omega)$ belongs to $H^{1,p}_0(\Omega)$, see \cite[Theorem~V.3.4 and Remark~V.3.5]{EE}. A converse to this result holds true under some regularity assumptions on $\partial \Omega$ (e.g. Lipschitz boundary is sufficient). Regularity is used to establish a Hardy inequality with weight $\delta^{-p}$, which implies $f/\delta \in L^p(\Omega)$ for $f \in H^{1,p}_0(\Omega)$, see \cite{KK66} and  also \cite{BEL15} for a more recent account on the subject.  Extensions of these results weakening the assumption $f/\delta \in L^p(\Omega)$  have recently gained some attention, see  \cite{KM97,EN17,NT22}. In this paper we extend Harris' criterion and its converse to discrete graphs. 

It turns out that precisely two properties of $\delta$ are required to establish Harris' criterion: One needs $|\nabla \delta| \leq 1$, which is a consequence of Rademacher's theorem, and that for all $\varepsilon > 0$ closed bounded subsets of $\{x \in \Omega \mid \delta(x) \geq \varepsilon\}$ are compact in $\Omega$. On graphs the first property is achieved by considering intrinsic metrics, while the second property will remain a mild assumption on the metric. Our discrete version of Harris' criterion is contained in Theorem~\ref{thm:Theorem 1} and Corollary~\ref{corollary:corollary  1}. 

For the converse to Harris' criterion we need to establish a Hardy inequality with respect to the distance to the boundary. Recently Hardy inequalities on discrete spaces have recieved quite some attention, see \cite{KPP18,KPP18a, FKP19,KPP20,KPP20a,KPP21,Fis22a,Fis22b,Fis22c}.  We rely on some methods developed in \cite{Fis22a,Fis22b,Fis22c} to obtain a Hardy inequality that yields the  converse to Harris' criterion in the discrete setting, see Theorem~\ref{thm:boundary hardy}, Corollary~\ref{corollary:corollary  2} and Corollary~\ref{corollary:corollary  3}. The  assumption  we have to make is some form of superharmonicity of the distance to the boundary. This is similar to the situation on Euclidean spaces, in which case superharmonicity of $\delta$ is related to convexity of the domain $\Omega$. We refer to \cite{BFT03} for how superharmonicity is employed to prove Hardy inequalities and how this is related to convexity of $\Omega$ when considering $\delta$. That the convexity of $\Omega$ can be replaced by regularity of the boundary $\partial \Omega$ is due to the fact that close to the boundary $\delta$ is equivalent to another superharmonic distance function as long as $\partial \Omega$ is sufficiently regular, see e.g. \cite{KK66}. It seems hopeless to extend this geometric reasoning to discrete spaces and hence we use superharmonicity  as an assumption. 



In this text we discuss discrete graphs but our results are valid in  more general situations. More precisely, our methods should yield versions of Harris' criterion  on regular Dirichlet spaces with suitable intrinsic metrics as in \cite{FLW} (in this case $p = 2$) and on suitable metric measure spaces with $p$-Cheeger energies as in \cite{GP20}.  In particular, the aforementioned Riemannian manifolds, metric graphs and some fractals are included in these settings.   Hardy inequalities  with respect to the distance to the boundary for a large class of linear local operators  were recently obtained \cite{KN23} and our results may help to establish them for certain non-local  operators as well.

We chose to focus on discrete graphs for a couple of reasons: The discreteness of the underlying space removes the need to discuss local regularity issues and within our framework  there is an easy explicit description of $H^{1,p}(X)$.  For regular Dirichlet spaces, which are  abstract versions of $H^{1,2}_0(X)$-spaces, there are constructions of $H^{1,2}(X)$-spaces via so-called reflected Dirichlet spaces, see \cite{Kuw,Schmi2}. Local regularity issues and the use of reflected Dirichlet spaces make technical details more lengthy but do not add insights on which type of results we could expect.  In contrast to this simplification, the non-local nature of discrete gradients and Laplacians makes proving Harris' criterion and Hardy inequalites more involved when it comes to certain estimates. This is due to the lack of a chain rule. Overcoming these difficulties in the discrete setting possibly paves the way for other non-local situations, e.g. fractional Laplace-type operators on various spaces.

Parts of this paper are based on SM Bachelor's thesis, whose main result is a preliminary version of Corollary~\ref{corollary:corollary  1}.

{\bf Acknowledgements.} M.S. thanks Florian Fischer and Matthias Keller for inspring discussions about Hardy inequalities. Moreover, M.S. acknowledges financial support of the DFG within the priority programme 'Geometry at Infinity'.

\section{Set Up and  Main Results}

Let $X \neq \emptyset$ be a countable set equipped with the discrete topology. Let $m \colon X \to (0,\infty)$. We abuse notation and also view $m$ as a measure on all subsets of $X$ by 
$$m(A) = \sum_{x \in A} m(x), \quad A \subseteq X. $$
Let $1 \leq p < \infty$. The discrete Lebesgue spaces are given by
$$\ell^p(X,m) = \{f \colon X \to \R \mid \sum_{x \in X} |f(x)|^p m(x) < \infty\}$$
equipped with the norm
$$\av{f}_p = \left(\sum_{x \in X} |f(x)|^p m(x)\right)^{1/p}.$$
The space of bounded functions is denoted by $\ell^\infty(X)$.  We write $C(X)$ for all real-valued functions on $X$. Due to $X$ being discrete the support of $f \in C(X)$ is given by ${\rm supp} f = \{x \in X \mid f(x) \neq 0\}$ and we write $C_c(X)$ for all real valued functions of finite support.  For $f,g \in C(X)$ we let $f \wedge g = \min \{f,g\}$ and $f \vee g = \max \{f,g\}$, which are defined pointwise.

A {\em weighted graph} on $X$ is a symmetric function $b \colon X \times X \to [0,\infty)$ such that $b(x,x) = 0$ and   
$$\deg(x) := \sum_{y \in X} b(x,y) < \infty$$
for each $x \in X$. Two points $x,y \in X$ are said to be {\em connected by an edge} if  $b(x,y) > 0$. In this case, we write $x \sim y$. The graph $b$ is called {\em locally finite} if for each $x \in X$ the set of its neighbors 
$$\{y \in X \mid x \sim y\}$$
is finite. A {\em path} connecting $x,y \in X$ is a finite  sequence $\gamma = (x_0,\ldots,x_n)$ in $X$ with  $x = x_0 \sim x_1 \sim \ldots \sim x_n = y$. The graph $b$ is called {\em connected} if all $x,y \in X$ are connected by a path.

The {\em formal $p$-Laplacian} is the nonlinear operator $\mathcal L \colon \mathcal F \to C(X)$ with domain
$$\mathcal F = \mathcal F_p = \{f \in C(X) \mid \sum_{y \in X} b(x,y) |f(x) - f(y)|^{p-1} < \infty \text{ for all } x \in X\}, $$
on which it acts by
$$\mathcal L f(x) = \frac{1}{m(x)} \sum_{y \in X} b(x,y) (f(x) - f(y))|f(x) - f(y)|^{p-2} .$$
Here we use the convention $0 /0 = 0$, which may occur in the definition of $\mathcal L$ whenever $1 \leq p < 2$. Note that $\ell^\infty(X) \subseteq \mathcal F_p$ due to the summability condinition on $b$. 
 
The {\em $p$-energy functional} is given by
$$\cQ \colon C(X) \to [0,\infty], \quad \cQ(f) = \frac{1}{p}\sum_{x,y \in X} b(x,y) |f(x) - f(y)|^p. $$
For $f \in C(X)$ we define the norm of the {\em discrete $p$-gradient} by
$$|\nabla f| \colon X \to [0,\infty],\quad  |\nabla f|(x) = \left(\frac{1}{m(x)} \sum_{y \in X} b(x,y) |f(x) - f(y)|^p\right)^{1/p}.$$
Clearly, the energy and the gradient are related through the formula
$$\cQ(f) = \frac{1}{p}\sum_{x \in  X} |\nabla f|^p(x) m(x).$$
We denote the space of {\em functions of finite $p$-energy} by 
$$\cD = \{f \in C(X) \mid \sum_{x,y \in X} b(x,y)|f(x) - f(y)|^p < \infty \} $$ 
and write 
$$H^{1,p}(X,m) = \cD \cap \ell^p(X,m)$$
for the corresponding first order {\em discrete Sobolev space}. It is a Banach space when equipped with the norm $\av{\cdot}_{H^{1,p}(X,m)}$ defined by 
$$\av{f}_{H^{1,p}(X,m)}^p = \mathcal Q(f) + \av{f}^p_p.$$
We write $H_0^{1,p}(X,m)$ for the closure of $C_c(X)$ in $H^{1,p}(X,m)$. 

The main goal of this paper is to characterize the space $H_0^{1,p}(X,m)$ within $H^{1,p}(X,m)$. To this end we employ the geometry induced by certain metrics that are compatible with the gradient. They are discussed next.

A {\em pseudo metric} on $X$ is a symmetric function $\sigma \colon X \times X \to [0,\infty)$ satisfying the triangle inequality. It is called {\em $p$-intrinsic} if 
$$\frac{1}{m(x)} \sum_{y \in X} b(x,y) \sigma(x,y)^p \leq 1$$
for every $x \in  X$. The most important feature of intrinsic pseudo metrics is the following Rademacher type result: If $\sigma$ is a $p$-intrinsic metric and $f \colon X \to \R$ is $L$-Lipschitz with respect to $\sigma$, then $|\nabla f| \leq L$.

\begin{remark}
In the case $p = 2$ intrinsic metrics were introduced in \cite{FLW} for regular Dirichlet forms mimicking the property that $1$-Lipschitz functions on Riemannian manifolds have distributional gradient bounded by $1$. We refer to \cite{Kel} for their use on graphs.  On spaces with a local Laplacian  (e.g. on Riemnnian manifolds) the norm of the gradient does not depend on $p$. In this case, one intrinsic metric can be employed for for all $1 \leq p < \infty$. In the discrete setting (and other non-local settings as well) intrinsic metrics may depend on the choice of $p$. 
\end{remark}

\begin{example}[Path pseudo metrics] \label{example:path metrics}
 Assume the graph $b$ is connected and let $w \colon X \times X \to [0,\infty)$ be symmetric. To a path $\gamma = (x_0,\ldots,x_n)$ we associate its {\em length} (with respect to $w$) by 
 $$L_w(\gamma) = \sum_{k = 1}^n w(x_{k-1},w_k).$$
 The function
 $$d_w \colon X \times X \to [0,\infty), \quad d_w(x,y) = \inf \{L_w(\gamma) \mid \gamma \text{ path connecting }x \text{ and }y\}$$
 is called the {\em path pseudo metric} induced by $w$. It is readily verified that $d_w$ is a pseudo metric. If the graph is locally finite and $w(x,y) > 0$ for all $x \sim y$, then $d_w$ is indeed a metric inducing the discrete topology on $X$.  For $w = 1$ the metric $d:= d_1$ is called {\em combinatorial distance}, as it counts the least number of edges in a path connecting $x,y$. 
 
 The path pseudo metric $d_w$ satisfies $d_w(x,y) \leq w(x,y)$. Hence, it is $p$-intrinsic if 
 $$\frac{1}{m(x)} \sum_{y \in X} b(x,y) w(x,y)^p \leq 1.$$
 One function $w$ satisfying this inequality is given by 
 $$w(x,y) = \frac{m(x)^{1/p}}{\deg(x)^{1/p}} \wedge  \frac{m(y)^{1/p}}{\deg(y)^{1/p}}.  $$
 The induced metric will be investigated in further examples below. 
 
 The combinatorial distance $d$ satisfies $d(x,y) = 1$ for $x \sim y$. Hence, it is $p$-intrinsic if and only if $ \deg/m$ is bounded by $1$. 
\end{example}

Let $\sigma$ be a pseudo metric. We write $\overline{X}^\sigma$ for the completion of $X$ with respect to $\sigma$ and denote by $\partial_\sigma X = \overline{X}^\sigma \setminus X$ the {\em Cauchy boundary} of $X$ with respect to $\sigma$ (if $\sigma$ is not a metric consider a suitable quotient space). Clearly, $\sigma$ can be uniquely extended to $\overline{X}^\sigma$. For $A \subseteq \overline{X}^\sigma$ we let 
$$\sigma(x,A) = \inf \{\sigma(x,y)\mid y \in A\} $$
with the convention $\sigma(x,A) = \infty$ if $A = \emptyset$. In what follows we will be interested in the behavior of the pseudo metric near the boundary. Therefore, we fix a cut-off $D > 0$ and introduce
$$\delta_\sigma \colon X \to [0,\infty),\quad \delta_\sigma(x) = \sigma(x,\partial_\sigma X) \wedge D, $$
where our convention implies $\delta_\sigma = D$ if $\partial_\sigma X = \emptyset$. Note that $\delta_\sigma(x) > 0$ for all $x \in X$ if and only if $\partial_\sigma X$ is closed in $\overline{X}^\sigma$.  The function $\ds$ is $1$-Lipschitz with respect to $\sigma$. Hence, if $\sigma$ is $p$-intrinsic, then $|\nabla \ds| \leq 1$.

The following is one  of the two main results of this paper.

\begin{theorem} \label{thm:Theorem 1}
 Let $1 < p < \infty$ and let $\sigma$ be a $p$-intrinsic pseudo metric.  Let $h \colon X \to (0,\infty)$  such that for all
$ \varepsilon > 0$ the function $\deg/m$ is bounded on $\sigma$-bounded subsets of $$\{x \in X \mid h(x) \geq \varepsilon\}.$$
Then 
$$\{f \in H^{1,p}(X,m) \mid \frac{|\nabla h|}{h} f \in \ell^p(X,m)\} \subseteq H_0^{1,p}(X,m).$$
\end{theorem}

\begin{remark}
\begin{enumerate}[(a)]
 \item  The possibility that  $|\nabla h|(x) = \infty$ for some $x \in X$ is allowed in previous theorem. In this case, the condition $\frac{|\nabla h|}{h} f \in \ell^p(X,m)$ is to be understood as $f(x) = 0$ for $x \in \{y \in X \mid |\nabla h|(y) = \infty\} =: N$ and $ 1_{X \setminus N} f |\nabla h|/h \in \ell^p(X,m)$.
 \item The theorem holds true for all $h \in C_0(X)$, which is the uniform closure of $C_c(X)$ in $\ell^\infty$. In this case, $\{x \in X \mid h(x) \geq \varepsilon\}$ is finite for any $\varepsilon > 0$. 
\end{enumerate}

\end{remark}

We can apply this theorem to the bounded $1$-Lipschitz function $h = \ds$ (and powers thereof). Since $\ds > 0$ if $\partial_\sigma X$ is closed and $|\nabla \ds| \leq 1$,  we obtain the following geometric criterion. 

\begin{corollary}[Harris' criterion for discrete spaces] \label{corollary:corollary  1}
 Let $1 < p < \infty$ and let $\sigma$ be a $p$-intrinsic pseudo metric with the following properties: 
 \begin{itemize} 
  \item $\partial_\sigma X$ is closed. 
  \item For all $ \varepsilon > 0$ the function $\deg/m$ is bounded on $\sigma$-bounded subsets of $$\{x \in X \mid \ds(x) \geq \varepsilon\}.$$ 
 \end{itemize}
Then for any $\alpha > 0$ we have
$$\{f \in H^{1,p}(X,m) \mid \frac{|\nabla \ds^\alpha|}{\ds^\alpha} f \in \ell^p(X,m)\} \subseteq H_0^{1,p}(X,m).$$
In particular, if $f \in  H^{1,p}(X,m)$ satisfies $f/\ds \in \ell^p(X,m)$, then $f \in H^{1,p}_0(X,m)$. 
\end{corollary}

As noted above, if  $\partial_\sigma X = \emptyset$, then $\delta_\sigma = D$. In this case,  the second assumption on $\sigma$ states that $\deg/m$ should be bounded on $\sigma$-balls.  Pseudo metric spaces with finite balls are always complete and $\deg/m$ is always bounded on finite sets. Hence, we obtain the following corollary.

\begin{corollary} \label{corollary:H = H0}
Let $1 < p < \infty$. If there exists a $p$-intrinsic pseudo metric with finite balls, then $H^{1,p}(X,m) = H^{1,p}_0(X,m)$.
\end{corollary}

Applying the theorem to $\sigma = 0$ and $h = 1$ we immediately obtain the following. 

\begin{corollary}
Let $1 < p < \infty$ and assume that $\deg/m$ is bounded. Then $H^{1,p}(X,m) = H^{1,p}_0(X,m)$.
\end{corollary}

\begin{remark}
 \begin{enumerate}[(a)]
  \item In the Euclidean setting mentioned in the introduction we have $|\nabla \delta| = 1$ almost surely and  $\nabla \delta^\alpha = \alpha \delta^{\alpha - 1} \nabla \delta$ by the chain rule.  Hence,  there is no gradient term and no exponent in Harris criterion, cf.  \cite[Theorem~V.3.4]{EE}.
  
  \item For $p = 2$ Corollary~\ref{corollary:H = H0} was first established in \cite{HKMW} under a slightly different assumption. We refer to \cite[Theorem~5.4]{Sch20} for related results. 
  \item The condition on $\deg/m$ used in Corollary~\ref{corollary:corollary  1} is fullfilled if bounded subsets of $\{x \in X \mid \ds(x) \geq \varepsilon\}$ are finite.  The latter condition along with some examples is discussed in \cite[Section~5]{Sch20}. We will give more examples in Section~\ref{section:examples}.
 \end{enumerate}

\end{remark}

Under a superharmonicity assumption on $\ds$ but without assuming $\sigma$ is intrinsic the converse of Harris' criterion holds true. Indeed, instead of considering $\ds$ one can actually use any strictly positive function. The following is the second main result of this paper.

\begin{theorem}[Hardy inequality]\label{thm:boundary hardy}
Let $1 <  p < \infty$ and let $h \colon X \to (0,\infty)$ with $h \in \mathcal F_p \cap \mathcal F_{p/2 + 1}$ such that there exists a finite $K \subseteq X$ and  $\lambda \in \R$ with  $\cL h \geq \lambda h^{p-1}$ on $X \setminus K$. Then there exists $C >0$ such that 
$$ C \sum_{x \in X} |f(x)|^p \frac{|\nabla h^{1/2}|^p(x)}{h(x)^{p/2}} m(x) \leq \mathcal Q(f) + \aV{f}_p^p$$
 for all $f \in H_0^{1,p}(X,m)$. In particular, 
$$H_0^{1,p}(X,m) \subseteq \{f \in H^{1,p}(X,m) \mid \frac{|\nabla h^{1/2}|}{h^{1/2}} f \in \ell^p(X,m)\}.$$
If, additionally,  $\sup\{h(x)/h(y) \mid x,y \in X,\, x \sim y\} < \infty$,  then there exists $C' > 0$ such that 
$$ C' \sum_{x \in X} |f(x)|^p \frac{|\nabla h|^p(x)}{h(x)^{p}} m(x) \leq \mathcal Q(f) + \aV{f}_p^p$$
 for all $f \in H_0^{1,p}(X,m)$.
\end{theorem}

\begin{remark}
\begin{enumerate}[(a)]
 \item  As mentioned in the introduction Hardy inequalities on discrete spaces have recently gained quite some attention with a focus on optimality of the weights, see \cite{KPP18a} for the case $p = 2$ and \cite{Fis22c} for general $1 < p < \infty$. Here we do not focus on optimality of constants/weights but  simply on the inclusion of Sobolev spaces.  In Proposition~\ref{proposition:hardy} we provide explicit constants under the somewhat stronger assumption  $\cL h \geq \lambda h^{p-1}$ on $X$. In the case $p = 2$ our Hardy inequality with weight $|\nabla h^{1/2}|^2/h$ is implicitly contained in \cite{KPP18a} but it seems that for $p \neq 2$ it is not (implicitly) contained in the literature. 
 \item The condition $h \in \mathcal F_p \cap \mathcal F_{p/2 + 1}$ is equivalent to $h \in  \mathcal F_p$ if $p \geq 2$ and $h \in \mathcal F_{p/2 + 1}$ if $p < 2$, see Lemma~\ref{lemma:formal domain} below. As discussed above, if $h$ is bounded, then it satisfies these conditions. 
\end{enumerate}
\end{remark}

Combining this theorem and Corollary~\ref{corollary:corollary  1} we obtain the following  characterizations of $H^{1,p}_0(X,m)$, which can be seen as discrete analogues of Harris' criterion and its converse.

\begin{corollary}\label{corollary:corollary  2}
    Let $1 < p < \infty$ and let $\sigma$ be a $p$-intrinsic pseudo metric with the following properties: 
 \begin{itemize} 
  \item $\partial_\sigma X$ is closed.
  \item For all $ \varepsilon > 0$ the function $\deg/m$ is bounded on bounded subsets of $$\{x \in X \mid \ds(x) \geq \varepsilon\}.$$ 
  \item There exists a finite $K \subseteq X$ and $\lambda \in \R$ such that  $\cL_p  \ds \geq \lambda \ds^{p-1}$ on $X \setminus K$.
  %
  %
  \end{itemize}
 Then
  $$H^{1,p}_0(X,m) = \{f \in H^{1,p}(X,m) \mid \frac{|\nabla \ds^{1/2} |}{\ds^{1/2}}  f \in \ell^p(X,m)\}.$$
\end{corollary}

\begin{corollary}\label{corollary:corollary  3}
Assume the situation of Corollary~\ref{corollary:corollary  2} and assume further:
\begin{itemize}
\item  
  $$\sup_{x,y \in X, x \sim y} \frac{\ds(x)}{\ds(y)} < \infty. $$
  \item $$\inf_{x \in X} |\nabla \ds|(x) > 0.$$
\end{itemize}
Then 
$$H^{1,p}_0(X,m) = \{f \in H^{1,p}(X,m) \mid    f/\ds \in \ell^p(X,m)\}.$$
\end{corollary}

\begin{remark}
 As already remarked above, the reason why in Corollary~\ref{corollary:corollary  2} we have to consider $|\nabla \ds^{1/2} | /\ds^{1/2}$ instead of $|\nabla \ds|/\ds$ is the lack of a chain rule. Instead of using a chain rule we can estimate $|\nabla \ds^{1/2}|$ from below with the help of the mean value theorem but then we have to assume $\sup\{h(x)/h(y) \mid x,y \in X,\, x \sim y\} < \infty$, see Lemma~\ref{lemma:pchain rule}. Moreover, in general we cannot expect $|\nabla \ds|$ to be bounded from below and so it appears as an assumption in Corollary~\ref{corollary:corollary  3}.
\end{remark}

%

%
%

\section{Some properties of discrete Sobolev spaces}

In this section we establish some fundamental results for the discrete Sobolev spaces and energy functionals. The results in this section are all well known in the case $p = 2$ and we refer to \cite[Section~1]{KLW21} for most of them.

\begin{lemma}[Contraction properties]\label{lemma:contraction properties}
Let $1 \leq p < \infty$.  Let $C \colon \R \to \R$ be $L$-Lipschitz. Then 
 $$\mathcal Q(C \circ f) \leq L^p \mathcal Q(f)$$
 for $f \in C(X)$.  In particular, $\mathcal Q(|f|) \leq \mathcal Q(f)$ and 
 $$\mathcal Q(f \wedge g)^{1/p} \leq \mathcal Q(f)^{1/p} + \mathcal Q(g)^{1/p} $$
 for all $f,g \in C(X)$. 
\end{lemma}
\begin{proof}
 The first inequality follows immediately from the definition of $\mathcal Q$.  The result on $|f|$ is a consequence of the absolute value being $1$-Lipschitz. For the last statement it suffices to consider $f,g \in \mathcal D$. Since $\mathcal Q^{1/p}$ is a seminorm on $\mathcal D$, it follows directly  from $f \wedge g = (f + g - |f-g|)/2$ and the result on the absolute value. 
\end{proof}

\begin{lemma}[Pointwise lower semicontinuity of $\mathcal Q$]
 Let $1 \leq p < \infty$. The functional $\mathcal Q$ is lower semicontinuous with respect to pointwise convergence, i.e., for all $(f_n)$, $f$ in $C(X)$ the convergence $f_n \to f$ pointwise, as $n \to \infty$, implies
 $$\mathcal Q(f) \leq  \liminf_{n \to \infty} \mathcal Q(f_n).$$
\end{lemma}
\begin{proof}
 This is a direct consequence of Fatou's lemma. 
\end{proof}

\begin{proposition} \label{prop:uniform convexity}
 For $1 < p < \infty$ the Banach spaces $H^{1,p}_0(X,m)$ and $H^{1,p}(X,m)$ are uniformly convex and hence reflexive.
\end{proposition}

\begin{proof}
Uniformly convex spaces are reflexive and so it suffices to prove uniform convexity. Let $E = \{(x,y) \in X \times X \mid b(x,y) > 0\}$ and let
$$\mu \colon X \cup E \to [0,\infty), \quad \mu(z) = \begin{cases}
                                                   \mu(x)&\text{if }z \in X  \\
                                                   b(z)/2p &\text{if }z \in E
                                                  \end{cases}.$$
It is well known that $\ell^p(X \cup E, \mu)$ is uniformly convex. With this at hand the uniform convexity of $H_0^{1,p}(X,m)$ and $H^{1,p}(X,m)$ follow from 
$$\iota \colon H^{1,p}(X,m) \to \ell^p(X \cup E,\mu),\quad \iota(f)(z) = \begin{cases}
                                                                          f(z) &\text{if } z \in X\\
                                                                          f(x) - f(y) &\text{if } z = (x,y) \in E
                                                                         \end{cases}
$$
being an isometry. 
\end{proof}
Next we give an alternative description of $H_0^{1,p}(X,m)$. We write $\mathcal D_0$ for the space of functions $f \in \mathcal D$ for which there exists a sequence $(\varphi_n)$ in $C_c(X)$ with $\mathcal Q(f - \varphi_n) \to 0$ and $\varphi_n \to f$ pointwise, as $n \to \infty$. 
\begin{proposition}\label{prop:H0 equals D0 cap ellp} 
  Let $1 < p < \infty$. Then $H^{1,p}_0(X,m) = \mathcal D_0 \cap \ell^p(X,m)$. 
\end{proposition}
\begin{proof}
 In our discrete setting convergence in $\ell^p(X,m)$ implies pointwise convergence. Hence, we have  $H^{1,p}_0(X,m) \subseteq \mathcal D_0 \cap \ell^p(X,m)$.
 
 $H^{1,p}_0(X,m) \supseteq \mathcal D_0 \cap \ell^p(X,m)$: Let $f \in D^p_0 \cap \ell^p(X,m)$.  Choose a sequence $(\varphi_n)$ in $C_c(X)$ with $\varphi_n \to f$ pointwise and $\mathcal Q(f- \varphi_n) \to 0$, as $n \to \infty$.  Consider the functions 
 $$\psi_n =  (\varphi_n \wedge |f|) \vee (-|f|).$$
 Then $\psi_n \in C_c(X) \subseteq H^{1,p}_0(X,m)$, $|\psi_n| \leq |f|$ and $\psi_n \to f$ pointwise. Lebesgue's dominated convergence theorem implies $\psi_n \to f$ in $\ell^p(X,m)$. Using the contraction property Lemma~\ref{lemma:contraction properties} yields 
 $$\mathcal Q(\psi_n)^{1/p} \leq \mathcal Q(\varphi_n)^{1/p} + 2 \mathcal Q(f)^{1/p}.$$
 This shows that $(\psi_n)$ is a bounded sequence in the Banach space $H^{1,p}_0(X,m)$. Using weak sequential compactness of bounded sets in reflexive spaces, we can assume without loss of generality $\psi_n \to g \in H_0^{1,p}(X,m)$ weakly in $H^{1,p}_0(X,m)$ (else consider a subsequence). Since in normed spaces weak and strong closures of convex sets coincide, we infer that for any $N \in \mathbb N$ we have
 $$ g \in \overline{{\rm conv} \{ \psi_n \mid n \geq N \}}^{H^{1,p}_0(X,m)}.  $$
 In particular, there exist finite convex combinations of $(\psi_n)$ of the form 
 $$g_N = \sum_{n = N}^{K_N} \lambda_n^{(N)} \psi_n  $$
 with $g_N \to g$ in $H^{1,p}_0(X,m)$, as $N \to \infty$. Since $\psi_n \to f \in \ell^p(X,m)$, we also have $g_N \to f$ in $\ell^p(X,m)$, showing $f = g \in H^{1,p}_0(X,m)$.  
\end{proof}

\begin{remark}
 For $p = 2$ this is the content of \cite[Theorem~1.19]{KLW21}.  
\end{remark}

Next we prove lower semicontinuity of $\mathcal Q$ when restricted to $\mathcal D_0$. More precisely, we let 
$$\mathcal Q_0 \colon C(X) \to  [0,\infty], \quad \mathcal Q_0(f) = \begin{cases}
                                                                         \mathcal Q(f) &\text{if } f \in \mathcal D_0\\
                                                                         \infty &\text{else}
                                                                        \end{cases}
$$
and show that this functional is lower semicontinuous. This is   more involved than lower semicontinuity of $\mathcal Q$, as it does not follow directly from Fatou's lemma.

\begin{proposition}\label{prop:lower semicontinuity}
 Let $1 < p < \infty$. The functional $\mathcal Q_0$ is lower semicontinuous with respect to pointwise convergence, i.e., for all $(f_n)$, $f$ in $C(X)$ the convergence $f_n \to f$ pointwise, as $n \to \infty$, implies
 $$\mathcal Q_0(f) \leq  \liminf_{n \to \infty} \mathcal Q_0(f_n).$$
 In particular, if $(\mathcal Q_0(f_n))$ is bounded, then $f \in \mathcal D_0$.
\end{proposition}
\begin{proof}
 It suffices to consider $(f_n)$ in $\mathcal D_0^p$ with $f_n \to f$ pointwise and 
 $$\liminf_{n \to \infty} \mathcal Q_0(f_n) = \lim_{n \to \infty} \mathcal Q_0(f_n) < \infty.$$
  By the definition of $\mathcal D_0$ we can choose a sequence $(\varphi_n)$ in $C_c(X)$ with $\mathcal Q(f_n - \varphi_n) \to 0$ and $\varphi_n \to f$ pointwise. In particular, the sequence $(\varphi_n)$ is $\mathcal Q$-bounded.

We denote by
$$\av{\cdot}_{\mathcal Q} \colon \mathcal D \to [0,\infty),\quad \av{f}_{\mathcal Q} = \mathcal Q(f)^{1/p} $$
the seminorm induced by the $p$-energy functional.  The space $(\mathcal D,\av{\cdot}_{\mathcal Q})$ is uniformly convex, cf. the proof of Proposition~\ref{prop:uniform convexity}. Hence, the completion of $(\mathcal D / \ker \mathcal Q,\av{\cdot}_{\mathcal Q})$ is reflexive. Using weak compactness of bounded sets in reflexive spaces and that convex sets have the same weak and strong closure, we infer the existence of finite convex combinations of $(\varphi_n)$ of the form 
 $$\psi_N = \sum_{n = N}^{K_N} \lambda^{(N)}_n \varphi_n$$
 such that $(\psi_N)$ is $\av{\cdot}_{\mathcal Q}$-Cauchy, cf. the argument given in the proof of Proposition~\ref{prop:H0 equals D0 cap ellp}. By definition $(\psi_N)$ belongs to $C_c(X)$.  Moreover,  $\varphi_n \to f$ pointwise, as $n \to \infty$, implies $\psi_N \to f$ pointwise, as $N \to \infty$.  The pointwise lower semicontinuity of $\mathcal Q$ yields 
 $$\mathcal Q(f - \psi_N) \leq  \liminf_{M \to \infty} \mathcal Q(\psi_M - \psi_N) \to 0, \text{ as } N \to \infty, $$
 and we obtain $f \in \mathcal D_0$. Since $\mathcal Q$ and $\mathcal Q_0$ agree on $\mathcal D_0$, the lower semicontinuity of $\mathcal Q$ implies 
 \begin{align}
  \mathcal Q_0(f) &= \mathcal Q(f) \leq \liminf_{n \to \infty} \mathcal Q(f_n) = \liminf_{n \to \infty} \mathcal Q_0(f_n). \hfill \qedhere
 \end{align}
\end{proof}

\begin{lemma} \label{lemma:sets of bounde deg}
 Let $1 \leq p < \infty$. Let $K \subseteq X$ such that $\deg/m$ is bounded on $K$. Then there exists a constant $C \geq 0$ such that 
 $$\mathcal Q (f 1_K) \leq C \sum_{x \in K} |f(x)|^p m(x)$$
 for all $f \in C(X)$. In particular, 
 $$\{f \in \ell^p(X,m) \mid \supp f \subseteq K\} \subseteq H^{1,p}_0(X,m).$$
\end{lemma}

\begin{proof}
 Let $f \in C(X)$ with ${\rm supp}f \subseteq K$. Using the symmetry of $b$ and the inequality $(|w| + |z|)^p \leq 2^p(|w|^p + |z|^p)$ we obtain 
 \begin{align*}
  p \mathcal Q(f) &= \sum_{x,y \in X} b(x,y) |f(x) - f(y)|^p \\
  &= \sum_{x,y \in K} b(x,y) |f(x) - f(y)|^p + 2 \sum_{x \in K} |f(x)|^p  \sum_{y \in X \setminus K} b(x,y)\\
  &\leq 2^{p+1} \sum_{x\in K} |f(x)|^p \sum_{y \in K} b(x,y) + 2 \sum_{x \in K} |f(x)|^p  \sum_{y \in X \setminus K} b(x,y)\\
  &\leq  C \sum_{x \in K} |f(x)|^p m(x).
 \end{align*}
 The statement on the discrete Sobolev spaces follows directly from this inequality and the density of $C_c(X)$ in $\ell^p(X,m)$.
\end{proof}

\section{Discrete chain rules and Hardy inequality}

In this section we collect some formulas replacing the chain rule in the discrete setting and use them to establish a Hardy inequality. 

We start with a lemma on the domain of the formal Laplacian that guarantees absolute convergence of the sums involved in the computations below.  

\begin{lemma}\label{lemma:formal domain}
 Let $1 \leq p,q < \infty$. 
 \begin{enumerate}[(a)]
  \item $\mathcal F_p = \{f \in C(X) \mid \sum_{y \in X} b(x,y)|f(y)|^{p-1} < \infty \text{ for all } x \in X\}. $
  \item $f \in C(X)$ satisfies $|\nabla f|(x) < \infty$ for all $x \in X$ if and only if $f \in \mathcal F_{p+1}$.
  \item If $p \leq q$, then $\mathcal F_q \subseteq \mathcal F_p$. 
 \end{enumerate}
  In particular,  we have $\ell^\infty(X) \subseteq \mathcal F_p$ for all $1 \leq p < \infty$ and if $f \in C(X)$ satisfies 
 $$\sup_{x,y \in X, x \sim y} \frac{|f(x)|}{|f(y)|} < \infty,$$
 then $f \in \mathcal F_p$ for all $1 \leq p < \infty$.
\end{lemma}
\begin{proof}
 (a): This follows directly from the summability condition on $b$ and the inequality $(|w| + |z|)^{p-1} \leq 2^p(|w|^{p-1} + |z|^{p-1})$. For more details see \cite[Lemma~2.1]{Fis22a}.
 
 (b): This follows directly from the definition of $|\nabla f|$ and $\mathcal F_{p+1}$.
 
 (c): This follows from the summability condition on $b$ and Hölder's inequality. 
 
 The 'In particular' statements are an immediate consequence of (a) and the summability condition on $b$.
\end{proof}

\begin{lemma}\label{lemma:pchain rule}
Let $f \in C(X)$ and let $I \subseteq \R$ be an interval with $f(X) \subseteq I$. If  $\varphi \in C^1(I)$ is increasing and concave, then  
$$|\nabla(\varphi \circ f)|^p(x) \geq \inf_{y \sim x} (\varphi' (f(x) \wedge f(y)))^p  |\nabla f|^p(x), \quad x \in X.  $$
%
%
\end{lemma}

\begin{proof}
The first inequality is a consequence of the mean value theorem and $\varphi'$ being decreasing due to the concavity of $\varphi$. 
\end{proof}

\begin{lemma} \label{lemma:comp}
  Let $f \colon X \to (0,\infty)$ with $f \in \mathcal F_{p+1}$.  Then $f \in \mathcal F_p$ and
 \begin{align*}
   2 f(x) &\mathcal L f(x) = |\nabla f|(x)^p \\
   &+ \frac{1}{m(x)} \sum_{y \in X} b(x,y) (f(x) + f(y))(f(x) - f(y)) |f(x) - f(y)|^{p-2}
  \end{align*}
 for all $x \in X$, with all sums converging absolutely. 
\end{lemma}

\begin{proof}
  That $f \in \mathcal F_p$ is a consequence of Lemma~\ref{lemma:formal domain} (b) and (c). It also guarantees absolut convergence of all the sums in the statement and in the computations below. We compute
 \begin{align*}
  f(x) \mathcal L f(x)  &=  \frac{1}{m(x)} \sum_{y \in X} b(x,y) (f(x)^2 - f(x) f(y)) |f(x) - f(y)|^{p-2}\\
  &=  \frac{1}{m(x)} \sum_{y \in X} b(x,y) (f(x) -  f(y))^2 |f(x) - f(y)|^{p-2} \\
  &\quad + \frac{1}{m(x)} \sum_{y \in X} b(x,y) (f(x) f(y) - f(y)^2) |f(x) - f(y)|^{p-2}.
 \end{align*}
Adding another $f \mathcal L f$ to both sides of this equation yields 
$$2f(x) \mathcal L f(x) = |\nabla f|^p(x) + \frac{1}{m(x)} \sum_{y \in X} b(x,y) (f(x)^2 - f(y)^2) |f(x) - f(y)|^{p-2},$$
which is the desired statement. 
\end{proof}

\begin{lemma} \label{lemma:main estimate} Let $f \colon X \to (0,\infty)$ with $f \in \mathcal F_p \cap \mathcal F_{p/2 + 1}$. Then $f^{1/2} \in \mathcal F_p$ and
$$ 2 f(x)^{1/2} \mathcal L (f^{1/2})(x) \geq  |\nabla f^{1/2}|^p(x) + \frac{1}{2^{p-2} f(x)^{(p-2)/2}} \mathcal L f(x).$$
%
 %
\end{lemma}
\begin{proof} Lemma~\ref{lemma:formal domain} ensures $f^{1/2} \in  \mathcal F_p$ such that all the quanties in the inequality are well-defined. We let $g = f^{1/2}$. As seen in Lemma~\ref{lemma:comp} we have
 \begin{align*}
  2 g(x) &\mathcal L g(x) = |\nabla g|(x)^p \\
   &+ \frac{1}{m(x)} \sum_{y \in X} b(x,y) (g(x) + g(y))(g(x) - g(y)) |g(x) - g(y)|^{p-2}.
 \end{align*}
We need to estimate the second summand. We let $\nabla_{xy}f  =  f(x) - f(y)$.  Using the mean value theorem we choose $\xi_{xy} \in [f(x) \wedge f(y), f(x) \vee f(y)]$ with
$$f(x)^{1/2} - f(y)^{1/2} = \frac{1}{2 \xi_{xy}^{1/2}} (f(x) - f(y)).$$
Our assumption  $f \in  \mathcal F_p \cap \mathcal F_{p/2+1}$ ensures that all sums in the following computation converge absolutely: 
\begin{align*}
  &\frac{1}{m(x)} \sum_{y \in X} b(x,y) (g(x) + g(y))(g(x) - g(y)) |g(x) - g(y)|^{p-2}\\
  &=  \frac{1}{2^{p-2} m(x)} \sum_{y \in X} b(x,y) (f(x) - f(y))  \frac{1}{\xi_{xy}^{(p-2)/2}}| f(x) - f(y)|^{p-2}\\
  &=  \frac{1}{2^{p-2} m(x)} \sum_{y \in X, f(x) \geq f(y)} b(x,y)  \frac{1}{\xi_{xy}^{(p-2)/2}} \nabla_{xy} f^{p-1} \\& \quad - \frac{1}{2^{p-2} m(x)} \sum_{y \in X, f(x) < f(y)} b(x,y)  \frac{1}{\xi_{xy}^{(p-2)/2}} \nabla_{yx} f^{p-1}  \\
     &\geq \frac{1}{2^{p-2} m(x)} \sum_{y \in X, f(x) \geq f(y)} b(x,y)  \frac{1}{f(x)^{(p-2)/2}} \nabla_{xy} f^{p-1} \\& \quad - \frac{1}{2^{p-2} m(x)} \sum_{y \in X, f(x) < f(y)} b(x,y)  \frac{1}{f(x)^{(p-2)/2}} \nabla_{yx} f^{p-1}\\
     &= \frac{1}{2^{p-2} f(x)^{(p-2)/2}} \mathcal L f(x). \hfill \qedhere
\end{align*}
%
%
 %
\end{proof}

The following lemma is the content of \cite[Lemma~3.5]{Fis22b}.

\begin{lemma}[Picone's inequality]
Let $1 < p < \infty$. Let $h \colon X \to (0,\infty)$ with $h \in \mathcal F_p$. For every $\varphi \in C_c(X)$ we have
$$\mathcal Q(\varphi) \geq \frac{2}{p}  \sum_{x \in X} \frac{\mathcal L h(x)}{h(x)^{p-1}} |\varphi(x)|^p m(x).$$
\end{lemma}

\begin{proposition}[Hardy inequality]  \label{proposition:hardy}
 Let $f \colon X \to (0,\infty)$ with $f \in \mathcal F_p \cap \mathcal F_{p/2 + 1}$ and assume there exists $\lambda \in \R$ with $\mathcal L f \geq \lambda f^{p-1}$. Then for all $\varphi \in C_c(X)$
 \begin{align*}
  \mathcal Q(\varphi) &\geq \frac{1}{p} \sum_{x \in X} \frac{|\nabla f^{1/2}|^p(x)}{f(x)^{p/2}} |\varphi(x)|^p m(x) +  \frac{\lambda}{p 2^{p-2}} \av{\varphi}_p^p.
 \end{align*}
 If, moreover, 
  $$K = \sup_{x,y \in X, \, x \sim y} \frac{f(x)}{f(y)} < \infty, $$
 then for all $\varphi \in C_c(X)$
 \begin{align*}
  \mathcal Q(\varphi) &\geq \frac{1}{p 2^p K^{p/2}} \sum_{x \in X} \frac{|\nabla f|^p(x)}{f(x)^p} |\varphi(x)|^p m(x) +  \frac{\lambda}{p 2^{p-2}} \av{\varphi}_p^p.
 \end{align*}
\end{proposition}
\begin{proof}
 By Lemma~\ref{lemma:formal domain} the assumption on $f$ yields $f^{1/2} \in \mathcal F_p$.  Applying Picone's inequality to $h = f^{1/2}$ shows
 $$\mathcal Q(\varphi) \geq \frac{1}{p}  \sum_{x \in X} \frac{2f(x)^{1/2} \mathcal L (f^{1/2})(x)}{f(x)^{p/2}} |\varphi(x)|^p m(x)$$
 for all $\varphi \in C_c(X)$. From Lemma~\ref{lemma:main estimate} we infer 
 \begin{align*}
  \frac{2f(x)^{1/2} \mathcal L (f^{1/2})(x)}{f(x)^{p/2}}  &\geq \frac{ |\nabla f^{1/2}|^p(x)}{f(x)^{p/2}}   + \frac{1}{2^{p-2} f(x)^{p - 1}} \mathcal L f(x)\\
  &\geq \frac{ |\nabla f^{1/2}|^p(x)}{f(x)^{p/2}}   + \frac{\lambda}{2^{p-2}}.
 \end{align*}
 Combining both inequalities shows the first claim. For the second inequality we use Lemma~\ref{lemma:pchain rule} to estimate
$$|\nabla f^{1/2}|^p(x) \geq \inf_{x \sim y} \frac{1}{2^p (f(x) \vee f(y))^{p/2}} |\nabla f|^p(x).   $$
Since by assumption $f(y) \leq K f(x)$ for all $x \sim y$, we  arrive at 
 \begin{align*}
  |\nabla f^{1/2}|^p(x) &\geq  \frac{1}{2^p K^{p/2} f(x)^{p/2}}   |\nabla f|^p(x). \hfill \qedhere
 \end{align*}
\end{proof}

\section{Proof of the main results}

In this section we prove the main results Theorem~\ref{thm:Theorem 1} and Theorem~\ref{thm:boundary hardy}.

\begin{proof}[Proof of Theorem~\ref{thm:Theorem 1}]
Let $f \in H^{1,p}(X,m)$ with $|\nabla h|/h f \in \ell^p(X,m)$ be given. According to Proposition~\ref{prop:H0 equals D0 cap ellp} it suffices to show $f \in \mathcal D_0^p$.

We fix $o \in X$. For $n \in \N$ we consider the function
$$\chi_n \colon X \to \mathbb R,  \quad \chi_n(x) = (h(x) - 1/n)_+ \cdot \left( (2 - \sigma(x,o)/n)_+ \wedge 1 \right).$$
Then $|\chi_n| \leq h$,  $\chi_n \to h$ pointwise, as $n \to \infty$, and 
$$ \supp \chi_n \subseteq B^\sigma_{2n}(o) \cap \{x \in X \mid h(x) \geq 1/n\}.$$
By our assumption on $\sigma$ the function $\deg/m$ is bounded on the latter set. 

Let $f_n = (\chi_n / h) f$. The support of $f_n$ is cointained in the support of $\chi_n$ and so Lemma~\ref{lemma:sets of bounde deg} implies $f_n \in H^{1,p}_0(X,m)$. Since $f_n \to f$ pointwise, by Proposition~\ref{prop:lower semicontinuity}  it suffices to show that $(\mathcal Q(f_n))$ is bounded for establishing $f \in \mathcal D_0^p$. 

For $g \in C(X)$ and $x,y \in X$ we write $\nabla_{xy} g = g(x) - g(y)$. We use the discrete product rule $\nabla_{xy}(gh) = g(x) \nabla_{xy} h + h(y) \nabla_{xy} g$  and the inequality $(|w| + |z|)^{p} \leq 2^p(|w|^{p} + |z|^{p})$ to estimate $p \mathcal Q(f_n)$ by 
\begin{align} \label{eq:inequality}
&\sum_{x,y \in X} b(x,y) |\nabla_{xy} f_n|^p \notag \\
 &= \sum_{x,y \in X} b(x,y)\left| \chi_n(x) \nabla_{xy} (f/h) + \frac{f(y)}{h(y)} \nabla_{xy} \chi_n  \right|^p \tag{$\heartsuit$}\\
 &\leq 2^p \sum_{x,y \in X} b(x,y)\left|\chi_n(x) \nabla_{xy} (f/h)\right|^p  + 2^p \sum_{x,y \in X} b(x,y) \left|\frac{f(y)}{h(y)} \nabla_{xy} \chi_n  \right|^p. \notag 
\end{align}

We use the discrete product rule and 
$$\nabla_{xy} \frac{1}{h} = \frac{1}{h(x)h(y)} \nabla_{yx} h$$
to estimate the first summand of the right hand side in Inequality~\ref{eq:inequality}: 
\begin{align*}
  &\sum_{x,y \in X} b(x,y)\left|\chi_n(x) \nabla_{xy} (f/h)\right|^p\\
 &= \sum_{x,y \in X} b(x,y) \left| \frac{\chi_n(x)}{h(x)} \nabla_{xy} f + \chi_n(x) f(y) \nabla_{xy}\frac{1}{h} \right|^p\\
 &= \sum_{x,y \in X} b(x,y) \left| \frac{\chi_n(x)}{h(x)} \nabla_{xy} f + \frac{\chi_n(x) f(y)}{h(x) h(y)}  \nabla_{yx}h \right|^p\\
 &\leq 2^p \sum_{x,y \in X} b(x,y)| \nabla_{xy} f |^p + 2^p \sum_{x,y\in X} b(x,y)  \frac{ |f(y)|^p}{|h(y)|^p}  |\nabla_{yx}h |^p \\
 &= p 2^p \mathcal Q(f) + 2^p \aV{f |\nabla h|/h}_p^p.
\end{align*}

Next we turn to the second summand of the right hand side in Inequality~\ref{eq:inequality}. We let $\rho \colon X \to \R, \, \rho(x) = (2 - \sigma(x,o)/n)_+ \wedge 1$. Then $|\rho | \leq 1$ and $\rho$ is $1/n$-Lipschitz, which yields $|\nabla \rho| \leq 1/n$. Using this and the discrete product rule we continue to estimate
\begin{align*}
&\sum_{x,y \in X} b(x,y) \left|\frac{f(y)}{h(y)} \nabla_{xy} \chi_n  \right|^p \\
&= \sum_{x,y \in X} b(x,y) \left|\frac{f(y) (h(y) - 1/n)_+}{h(y)}  \nabla_{xy} \rho + \frac{f(y)}{h(y)} \rho(x) \nabla_{xy} (h - 1/n)_+   \right|^p\\
&\leq 2^p \sum_{x,y \in X} b(x,y) |f(y)|^p   |\nabla_{xy} \rho|^p + 2^p \sum_{x,y \in X} b(x,y) \frac{|f(y)|^p}{|h(y)|^p} |\nabla_{xy} (h - 1/n)_+|^p\\
&\leq 2^p \sum_{y \in X} |f(y)|^p m(y) |\nabla \rho|^p(y) + 2^p \sum_{y \in X} |f(y)|^p \frac{|\nabla h|^p(y)}{|h(y)|^p} m(y)\\
&\leq 2^p \av{f}_p^p + 2^p \aV{f |\nabla h|/h}_p^p.
\end{align*}
For the second to last inequality we  used $|\nabla_{xy} (h -1/n)_+| \leq |\nabla_{xy} h|.$

Combining all these estimates yields a constant $C > 0$ such that 
$$C \mathcal Q(f_n) \leq \mathcal Q(f) +  \av{f}_p^p +  \aV{f |\nabla h|/h }_p^p. $$
By our assumption on $f$ the right hand side of this inequality is finite and we arrive at the desired conclusion. 
\end{proof}

\begin{proof}[Proof of Theorem~\ref{thm:boundary hardy}] 
By assumption there exists a finite $K \subseteq X$  with  $\cL h \geq \lambda h^{p-1}$ on $X \setminus K$. Since $h (x) > 0$ for all $x \in X$ and $\mathcal L h$ is bounded below on the finite set $K$,  there exists $\lambda' \in \R$ with  $\mathcal L h \geq \lambda' h$ on $K$. Hence, we infer $\mathcal L h \geq (\lambda \wedge \lambda') h$ on $X$. With this at hand the first inequality in Proposition~\ref{proposition:hardy} yields the existence of $C > 0$ such that
$$ C \sum_{x \in X} |\varphi(x)|^p \frac{|\nabla h^{1/2}|^p(x)}{h(x)^{p/2}} m(x) \leq \mathcal Q(\varphi) + \aV{\varphi}_p^p$$
for all $\varphi \in C_c(X)$. By definition  $H^{1,p}_0(X,m)$ is the closure of $C_c(X)$ in $H^{1,p}(X,m)$. Since the left-hand side of the above inequality is lower semicontinuous with respect to pointwise convergence and since $\av{\cdot}_{H^1(X,m)}$-convergence implies pointwise convergence, the inequality extends to $\varphi \in H^{1,p}_0(X,m)$. The same reasoning with the use of the second inequality in Proposition~\ref{proposition:hardy} shows the claim. 
\end{proof}

\section{Spherically symmetric trees} \label{section:examples}
In this section we apply our main results to spherically symmetric trees. We start with a short recap on spherically symmetric graphs, then provide a discussion on boundaries of trees in general  and finally apply everything to certain spherically symmetric trees.

\subsection{Weakly spherically symmetrics graphs}

Let $b$ be a graph on $X$ with combinatorial metric $d:= d_1$, which was introduced in Example~\ref{example:path metrics}, and fix $o \in X$. A function  $f \colon X \to \R$ is called spherically symmetric (with respect to $o$) if $f(x) = f(y)$ whenever $d(x,o) = d(y,o)$. In this case, there exists $\tilde f \colon \mathbb N_0 \to \R$ such that $f(x) = \tilde f (n)$ for $x \in X$ with $d(x,o) = n$. In what follows we abuse notation and do not distinguish between $\tilde f$ and $f$. 

For $n \in \mathbb N_0$ we write $S_n = \{x \in X \mid d(o,x) = n\}$ for the combinatorial spheres around $o$. The incoming and outgoing degree functions $\deg_-$ and $\deg_+$ are defined by 
$$\deg_{\pm} \colon X \to [0,\infty),\quad \deg_{\pm}(x) = \sum_{y \in S_{d(x,o) \pm 1}} b(x,y),  $$
with the convention $\deg_-(o) = 0$. 
We call a graph $b$ with measure $m$ {\em weakly spherically symmetric} with respect to $o$ if the functions $\kappa_\pm := \deg_{\pm}/m$ are spherically symmetric with respect to $o$. We will only deal with the special case where both $\deg_{\pm}$ and $m$ are spherically symmetric. 

The important property of weakly spherically symmetric graphs is that whenever $f \in \mathcal F_p$ is spherically symmetric, then $\mathcal L f$ is also spherically symmetric with
\begin{align*}
 \mathcal Lf (n) &=  \kappa_+(n) (f(n) - f(n+1))|f(n) - f(n+1)|^{p-2} \\
 & \quad + \kappa_-(n) (f(n) - f(n-1))|f(n) - f(n-1)|^{p-2} 
\end{align*}
if $n \geq 1$ and 
$$\mathcal Lf (0) =    \kappa_+(0)(f(0) - f(1))|f(0) - f(1)|^{p-2}. $$
For $f \in C(X)$ we let 
$$A f \colon \N_0 \to \R, \quad A(n) = \sum_{x \in S_n} f(x). $$
 For spherically symmetric $m$ we then have
$$\av{f}_p^p = \sum_{n = 0}^\infty A|f|^p (n) m(n).$$

We refer to \cite{KLW} for more details in the case when $p = 2$. 

\subsection{Trees and metric boundaries} Let  $b$ on $X$ be a {\em connected tree}, i.e., between any two different points in $X$ there is a unique injective path connecting them. Moreover, we assume  $b(x,y) \in \{0,1\}$ for all $x,y \in X$. This and our summability condition on $b$ yield that $b$ is locally finite. We fix a root $o \in X$. It is readily verified that $b$ being  a tree is equivalent to $\deg_-(x) = 1$ for all $x \in X\setminus \{o\}$.
 
 We describe the metric boundary of a tree with respect to a path metric.  On a tree we say that $z \in X$ is an {\em ancestor} of $x \in X$ (with respect to $o$)  if every path from $x$ to $o$ an passes through $z$. Any set $\emptyset \neq A \subseteq  X$ possesses a unique {\em greatest common ancestor} $\wedge A$, which is an ancestor of all $x \in A$ and if $z \in X$ is another ancestor of all $x \in A$, then $d(z,o) \leq d(\wedge A,o)$. Below we simply write $x \wedge y$ for $\wedge \{x,y\}$.
 
 Let $w \colon X \times X \to [0,\infty)$ with $w(x,y) > 0$ for all $x \sim y$. As discussed in Example~\ref{example:path metrics} the corresponding path metric $d_w$ induces the discrete topology on $X$. Hence, as the union of the open sets $\{x\}, x \in X$, the set $X$ is open in the completion and so $\partial_{d_w} X = \overline{X}^{d_w} \setminus X$ is closed.  For computing the metric boundary with respect to $d_w$ we use two properties of path metrics on trees:  For $x,y \in X$ we have
 $$d_w(x,y) = d_w(x,x\wedge y) + d_w(x \wedge y,y) = L_w(\gamma_{x,x \wedge y}) + L_w(\gamma_{x \wedge y, y}), $$
 where $\gamma_{w,z}$ denotes the unique injective path  connecting $w,z \in X$. Moreover, if $x,y,z \in X$ satisfy $x \wedge y = x$ and $y \wedge z   = y$,  then 
 $$d_w(x,z) = d_w(x,y) + d_w(y,z).$$

 \begin{lemma}\label{lemma:monotone paths}
 Let $d_w$ be a path metric as above such that $\partial_{d_w} X \neq \emptyset$.  For every $x \in \partial_{d_w}  X$ there exists an infinite  injective path $(x_n)$ with limit $x$ and $x_n \wedge x_m = x_m$ for $m \geq n$. In particular, for $n > m$ we have
 $$d_w(x_n,x_m) = \sum_{k = m}^{n-1} w(x_k,x_{k+1}).$$
 \end{lemma}
\begin{proof}
This is contained in the proof of \cite[Proposition~3.7]{LPS23}. The 'In particular' part follows by induction from the previously mentioned properties of path metrics on trees.
\end{proof}

Let $\gamma = (x_n)$ be an infinite path. As for finite paths we define its length by
$$L_w(\gamma) = \sum_{n = 1}^\infty w(x_n,x_{n+1}) \in [0,\infty].$$
%

\begin{proposition}\label{proposition:distance to boundary trees}
 Let $d_w$  be a path metric as above. The following assertions are equivalent.
 \begin{enumerate}[(i)]
  \item $\partial_{d_w} X = \emptyset$.
  \item All infinite paths have infinite lenght with respect to $w$.
  \item All balls with respect to $d_w$ are finite. 
 \end{enumerate}
 Moreover, for $x \in X$ we have  
 $$d_w(x,\partial_{d_w}X) = \inf \{L_w (\gamma_x) \mid \gamma_x \text{ an infinite path starting at }x\}.$$
\end{proposition}
\begin{proof}
The statement on completeness is a  consequence of a more general theorem. On locally finite graphs (i), (ii) and (iii) are equivalent for path metrics, see \cite[Appendix~A]{HKMW}. This implies  the formula for $d_w(x,\partial_{d_w}X)$ whenever $\partial_{d_w}X = \infty$.

Now assume $\partial_{d_w}X \neq \emptyset$ and let 
$$I_x := \inf \{L_w (\gamma_x) \mid \gamma_x \text{ an infinite path starting at }x\}.$$

$d_w(x,\partial_{d_w}X) \leq I_x:$ Let $\gamma_x = (x_n)$ be an infinite path starting in $x$. Without loss of generality we can assume $L_w(\gamma_x) < \infty$. Then $\gamma_x$ must eventually leave every finite set. Moreover, by triangly inequality we have for $n \geq m$
$$d_w(x_n,x_m) \leq \sum_{k = m}^{n-1} w(x_k,x_{k+1}) \leq \sum_{k = m}^{\infty}w(x_k,x_{k+1}),$$
which tends to $0$, as $m \to \infty$. Hence, $(x_n)$ is $d_w$-Cauchy with limit $y \in \overline{X}^{d_w}$. Since the sequence eventually leaves every finite set, we obtain $y \in \partial_{d_w}X$ and
$$d_w(x,y) = \lim_{n \to \infty} d_w(x,x_n) \leq \lim_{n \to \infty}   \sum_{k = 1}^{n-1} w(x_k,x_{k+1}) = L_w(\gamma_x). $$
This shows $d_w(x,\partial_{d_w}X) \leq I_x$.

$d_w(x,\partial_{d_w}X) \geq I_x$: Let $\varepsilon > 0$. Choose $y \in \partial_{d_w} X$ with $d_w(x,\partial_{d_w}X) \geq d_w(x,y) - \varepsilon$. Let $\gamma = (y_n)$ be a path converging to $y$ with $y_n \wedge y_m = y_m$ for $m \geq n$ as in Lemma~\ref{lemma:monotone paths}.

Case 1: $x = y_m$ for some $m \in \N$. 

Let $m \in \N$ with $x = y_m$ and consider the path $\gamma_x = (y_m,y_{m+1},\ldots)$. Then 
$$d_w(x,y) = \lim_{n \to \infty} d_w(y_m,y_n) = \lim_{n \to \infty} \sum_{k = m}^{n-1} w(x_k,x_{k+1}) = L_w(\gamma_x).$$

Case 2: $x$  is not contained in $\gamma$. 
 
The sequence $(y_n)$ leaves every finite set. Hence, there exists $m \in \N$ with $d(x,o) \leq d(y_m,o)$. Let $ \gamma' = (x,x_1,\ldots,x_n,y_m)$ be the shortest path connecting $x$ and $y_m$. Since $d(x,o) \leq d(y_m,o)$ and since $y_n \wedge y_m = y_m$ for $n \geq m$, this path does not pass through any other $y_n,  n > m$. Hence,  $\gamma_x = (x,x_1,\ldots,x_n,y_{m},y_{m+1},\ldots)$ is a geodesic and we obtain 
$$d_w(x,y) = \lim_{n \to \infty} d_w(x,y_n) = L_w(\gamma') +  \lim_{n \to \infty} d_w(y_m,y_n)  = L_w(\gamma_x). $$
Combining both cases yields the claim.
\end{proof}

 \subsection{Spherically symmetric trees} Let $b$ be as in the previous subsection.   We additionally assume that $b$ is a {\em spherically symmetric tree} with branching numbers $k\colon \mathbb N_0 \to \mathbb N$, i.e., $\deg_-(x) = 1$ for all $x \in X \setminus \{o\}$ and $\deg_+(x) = k(n)$ if $d(x,o) = n$. Then $\deg(x) = k(n) + 1$ if $n = d(x,o) \geq 1$ and $\deg(o) = k(0)$.   Furthermore, we let $m \colon X \to (0,\infty)$ be spherically symmetric and consider the weights
 $$w(x,y) = \frac{m(x)^{1/p}}{\deg(x)^{1/p}} \wedge  \frac{m(y)^{1/p}}{\deg(y)^{1/p}},$$
 which were introduced Example~\ref{example:path metrics}. We now compute all the quantities relevant for applying our main results (more precisely their culmination in Corollary~\ref{corollary:corollary  3}) to the function $\delta_{d_w}$.

We fix $x \in X$ with $d(x,o) = n \in \N$. The case $n = 0$ can be treated similarly but will not really matter in computations. Let $y \in X$ with $d(y,o) = n + 1$. Then 
 $$w(x,y) = \frac{m(n)^{1/p}}{(k(n)+ 1) ^{1/p}} \wedge  \frac{m(n+1)^{1/p}}{(k(n+1) + 1)^{1/p}} =: \alpha(n). $$
The spherical symmetry implies that a shortest infinite path  $\gamma_x = (x_0,x_1,\ldots)$ with respect to $w$ starting in $x$ always increases its combinatorial distance to the root, i.e., $d(x_{k+1},o) = d(x_k,o) + 1$ for all $k \in \N$. Iterating this yields $d(x_k,o) = d(x,o) + k = n + k$ and our formula for $w$ yields
$$L_w(\gamma_x) =   \sum_{k = 0}^\infty w(x_k,x_{k+1}) = \sum_{k = n}^\infty  \alpha(k).  $$
Using Proposition~\ref{proposition:distance to boundary trees} this shows $\partial_{d_w} X \neq \emptyset$ if and only if $\alpha \in \ell^1(\N)$. In particular, we obtain the following general result.

\begin{proposition}\label{prop:h=h_0 trees}
Let $1 < p < \infty$. Assume the situation described above with $\alpha \not \in \ell^1(\N)$. Then $H^{1,p}(X,m) = H_0^{1,p}(X,m)$. 
\end{proposition}
\begin{proof}
 As discussed, $\alpha \not\in \ell^1(\N)$ leads to $\partial_{d_w} X = \emptyset$, which according to Proposition~\ref{proposition:distance to boundary trees} is equivalent to $d_w$-balls being finite. We infer $H^{1,p}(X,m) = H_0^{1,p}(X,m)$ from Corollary~\ref{corollary:H = H0}.
\end{proof}

Now we assume the latter case $\alpha \in \ell^1(\N)$. Letting  $D = \sum_{k=1}^\infty \alpha(k)$ we obtain 
$$\delta_{d_w}(x) = d_w(x, \partial_{d_w} X) \wedge D = \sum_{k=n}^\infty \alpha(k),  $$
showing that $\delta_{d_w}$ is spherically symmetric. Since $\delta_{d_w}(n) \to 0$, for $n \to \infty,$ the set $\{z \in X  \mid \delta_{d_w}(z) \geq \varepsilon\}$ is finite for any $\varepsilon > 0$. For $y \sim x$ we either have $d(y,o) = n - 1$ or $d(y,o) = n+1$. Since  $\delta_{d_w}(n) \leq \delta_{d_w}(n-1)$ and
$$\frac{\delta_{d_w}(n)}{\delta_{d_w}(n + 1)} =  \frac{\alpha(n) + \delta_{d_w}(n+1)}{\delta_{d_w}(n + 1)} = 1 + \frac{\alpha(n)}{\delta_{d_w}(n + 1)}, $$
we have
$$\sup_{y,z \in X, y\sim z} \frac{\delta_{d_w}(y)}{\delta_{d_w}(z)} < \infty$$
if and only if $\sup_n \alpha(n)/\delta_{d_w}(n+1) < \infty$. 

Proposition~\ref{prop:h=h_0 trees} shows that $H^{1,p}(X,m) \neq H^{1,p}_0(X,m)$ implies $m$ being small relative to $k$. Hence, for simplicity  we assume from now on that $n \mapsto m(n)$ is decreasing and $n \mapsto k(n)$ is increasing, which leads to $\alpha(n) = m(n+1)^{1/p}/(k(n+1) + 1)^{1/p}$. The gradient of $\delta_{d_w}$ is then given by
\begin{align*}
 |\nabla \delta_{d_w}|^p(n) &= \frac{k(n)}{m(n)}|\delta_{d_w}(n) - \delta_{d_w}(n+1)|^p  +  \frac{1}{m(n)}|\delta_{d_w}(n) - \delta_{d_w}(n-1)|^p\\
 &= \frac{k(n)\alpha(n)^p}{m(n)} +   \frac{\alpha(n-1)^p}{m(n)}\\
 &= \frac{k(n)m(n+1)}{m(n)(k(n+1) + 1)} + \frac{1}{k(n) + 1}, 
\end{align*}
and we have
\begin{align*}
 \mathcal L \delta_{d_w}(n) &= \frac{k(n)}{m(n)} (\delta_{d_w}(n) - \delta_{d_w}(n+1))^{p-1} - \frac{1}{m(n)} (\delta_{d_w}(n-1) - \delta_{d_w}(n))^{p-1}\\
 &= \frac{k(n)\alpha(n)^{p-1}}{m(n)} - \frac{\alpha(n-1)^{p-1}}{m(n)}\\
 &=  \frac{k(n)m(n+1)^{(p-1)/p}}{m(n)(k(n+1) + 1)^{(p-1)/p}} -  \frac{1}{(k(n) + 1)^{(p-1)/p}m(n)^{1/p}}.
\end{align*}
For the first equality above we used that $n \mapsto \delta_{d_w}(n)$ is  decreasing.  
%
%
%

We now apply Corollary~\ref{corollary:corollary  3} in two concrete examples, where $k$ and $m$ are polynomial, respectively exponential functions. Given two function $f,g \colon \N \to \R$ we say {\em $f(n) \asymp g(n)$ for $n \geq K$} if there exist constants $C,D > 0$ such that $C f(n) \leq g \leq D f(n)$ for all $n \geq K$.

\begin{example}[Polynomial growth]
Let $1 < p < \infty$. Let $\gamma  > 0 > \eta$ and assume $k$ is increasing with $k(n) \asymp n^\gamma$ for $n \geq 1$ and  $m$ is decreasing with $m(n) \asymp n^\eta$ for $n \geq 1$. If $\gamma \leq \eta + p$, then $H^{1,p}(X,m) = H^{1,p}_0(X,m)$. If $\gamma > \eta + p$, then there exist $C> 0$ such that 
$$C \sum_{n = 1}^\infty A |f|^p (n) n^{\gamma - p}  \leq  \mathcal  Q(f)  + \sum_{n = 1}^\infty A |f|^p (n) n^{\eta} + f(0)^p m(0)$$
for all $f \in  H^{1,p}_0(X,m)$. In particular,  $f \in H^{1,p}_0(X,m)$ if and only if $f \in \mathcal D$ and 
$$\sum_{n = 1}^\infty A |f|^p (n) n^{\gamma - p} < \infty.$$
\end{example}
\begin{proof}
Our assumptions yield $\alpha(n) \asymp \frac{1}{n^{(\gamma - \eta)/p}}$ for $n \geq 1$. Hence, we have $\alpha  \in \ell^1(\N)$ 
%
%
%
if and only if $\gamma > \eta + p$.  In this case, we can estimate 
$$\delta_{d_w}(n) = \sum_{l=n}^\infty \alpha(l) \asymp   \sum_{l=n}^\infty \frac{1}{l^{(\gamma - \eta)/p}} \asymp \frac{1}{n^{(\gamma - \eta)/p -1 }},  $$
which implies 
$$\sup_n \frac{\alpha(n)}{\delta_{d_w}(n+1)} < \infty.   $$
For the gradient we obtain 
\begin{align*}
 |\nabla \delta_{d_w}|^p(n) &\asymp \frac{n^\gamma (n + 1)^\eta}{n^\eta ((n+1)^\gamma + 1)} + \frac{1}{n^\gamma + 1} \asymp 1. 
\end{align*}
For estimating the Laplacian we find constants  $C,C',D > 0$ such that for $n \in \N$ we have
\begin{align*}
 \mathcal L \delta_{d_w}(n) &\geq C \frac{n^\gamma (n+1)^{\eta (p-1)/p}}{n^\eta((n+1)^\gamma + 1)^{(p-1)/p}} -  D \frac{1}{(n^\gamma + 1)^{(p-1)/p}n^{\eta/p}}\\
 &\geq C' n^{(\gamma - \eta)/p} - D.
\end{align*}
Since $\gamma > \eta + p$, this shows  $\mathcal L \delta_{d_w} \geq 0$ outside some finite set. Now the claims follow from Corollary~\ref{corollary:corollary  3} and our discussion prior to this example.
\end{proof}

\begin{example}[Exponential growth]
 Let $\gamma > 1 \geq  \eta > 0$ and assume $k$ is increasing with $k(n) \asymp \gamma^n$  for $n \geq 0$ and $m$ is decreasing with $m(n) \asymp \eta^n$ for $n \geq 0$. Then there exists $C > 0$ such that 
 $$C \sum_{n = 0}^\infty A |f|^p (n)  \gamma^n  \leq  \mathcal  Q(f)  + \sum_{n = 0}^\infty A |f|^p (n) \eta^{n}$$
 for all $f \in H^{1,p}_0(X,m)$.  Moreover,  $f \in H^{1,p}_0(X,m)$ if and only if $f \in \mathcal D$ and 
$$\sum_{n = 1}^\infty A |f|^p (n) \gamma^n < \infty.$$
 
\end{example}

\begin{proof}
 Let $\xi = (\eta/\gamma)^{1/p}$ such that $\alpha(n) \asymp \xi^n$ for $n \geq 0$. Since $\eta < \gamma$, we obtain  $\alpha \in \ell^1(\N)$ and 
 $$\delta_{d_w}(n) \asymp \sum_{l=n}^\infty \frac{\eta^{l/p}}{(\gamma^l + 1)^{1/p}} \asymp \frac{\xi^n}{1-\xi}, \quad n \geq 0. $$
 This implies $\sup_n \frac{\alpha(n)}{\delta_{d_w}(n+1)} < \infty$. For the gradient of $\delta_{d_w}$ we obtain 
 \begin{align*}
  |\nabla \delta_{d_w}|^p(n) &\asymp \frac{\gamma^n \eta^{n+1}}{\eta^n (\gamma^{n+1} + 1)} + \frac{1}{\gamma^n + 1} \asymp 1, \quad n \geq 1. 
 \end{align*}
For estimating the Laplacian we find constants $C,C',D,D'  > 0$ such that  for $n \geq 1$
 \begin{align*}
    \mathcal L \delta_{d_w}(n) &\geq C \frac{\gamma^n \eta^{(n+1)(p-1)/p}}{\eta^n (\gamma^{n+1} + 1)^{(p-1)/p}} - D \frac{1}{(\gamma^n + 1)^{(p-1)/p} \eta^{n/p}}\\
    &\geq C' \xi^{-n}  - \frac{D'}{\gamma^n} \xi^{-n}.
 \end{align*}
 Since $\xi < 1$ and $\gamma > 1$, we infer $\mathcal L   \delta_{d_w} \geq 0$ outside a finite set. Now the claims follow from Corollary~\ref{corollary:corollary  3} and our discussion prior to this example.
\end{proof}

\begin{remark}
 Similar results hold true for other classes of spherically symmetric graphs. Most notably, spherically symmetric anti-trees can be treated as above. We  briefly recall their definition and leave the details to the reader. 
 
 A connected graph $b$ over $X$ with fixed $o \in X$ and $b(x,y) \in  \{0,1\}$ for all $x,y \in X$ is called {\em spherically symmetric anti-tree} if $x \sim y$ whenever $x \in S_n$ and $y \in S_{n + 1}$ or $y \in S_n$ and $x \in S_{n+1}$ for some $n \in \N_0$.  In other words, anti-trees are those connected graphs in which all points from $S_n$ are connected to all points from $S_{n+1}$. They possess the largest number of edges between $S_n$ and $S_{n+1}$ as possible (namely $|S_n| \cdot |S_{n+1}|$). This is opposite to connect trees, who possess the least number of edges between $S_{n}$ and $S_{n+1}$ without loosing  connectedness (namely $|S_{n+1}|$). The anti-tree with $|S_n| = n + 1, n \geq 0,$ appeared in \cite{DK88,Web10} as an example.  General anti-trees were introduced in \cite{Woj11} to provide examples of stochastically incomplete graphs of  polynomial volume growth with respect to the combinatorial metric. 
 
 Spherically symmetric anti-trees satisfy $\deg_+(n) = |S_{n + 1}|$ for $n \geq 0$ and $\deg_-(n) =  |S_{n-1}|$ for $n \geq 1$.  Equipping them with a spherically symmetric $m$ and the $p$-intrinsic metric $d_w$ discussed above, $\partial_{d_w} X$ consists of at most one point (use that points in consecutive spheres are connected) and for $n \geq 1$ we have
 $$d(n,\partial_{d_w}X) = \sum_{l = n}^\infty \frac{m(l)^{1/p}}{ (|S_{l-1}|  + |S_{l + 1}|)^{1/p}} \wedge  \frac{m(l+1)^{1/p}}{ (|S_{l}|  + |S_{l + 2}|)^{1/p}}.   $$
 With this at hand similar results as for trees can be established using the formulas for the norm of the gradient and for the Laplacian on spherically symmetric graphs. We refrain from giving further details. 
\end{remark}

\bibliographystyle{plain}
 
\bibliography{literatur}

\end{document}